\documentclass[11pt]{amsart}
\usepackage[utf8]{inputenc}
\usepackage{comment}
\usepackage[hidelinks]{hyperref}
\usepackage{graphicx} 

\usepackage[nobysame,abbrev]{amsrefs}

\usepackage{amsmath, amssymb, amsfonts}
\usepackage{subcaption} 

\newtheorem{theorem}{Theorem}
\newtheorem{lemma}[theorem]{Lemma}
\newtheorem{proposition}[theorem]{Proposition}

\newtheorem{remark}{Remark}

\newcommand{\floor}[1]{\left\lfloor #1 \right\rfloor}
\newcommand{\ceil}[1]{\left\lceil #1 \right\rceil}

\newcommand{\eps}{\varepsilon}
\newcommand{\Prob}{\mathbb{P}}
\newcommand{\ind}[1]{\mathbf{1}_{#1}}
\newcommand{\prob}[1]{\Prob{\left( #1 \right)}}

\newcommand{\prd}[2]{\Prob_{\mathrm{D}_{#1}}{#2}}
\newcommand{\prb}[2]{\Prob_{\mathrm{B}_{#1}}{#2}}
\newcommand{\prdp}[1]{\prd{p}{\left(#1\right)}}
\newcommand{\prbp}[1]{\prb{p}{\left(#1\right)}}
\newcommand\Bin{\operatorname{Bin}}

\newcommand\Var{\operatorname{Var}}
\newcommand\Cov{\operatorname{Cov}}

\newcommand\E{\mathbb{E}}
\newcommand\Ec[2]{{\E\left(#1\,|\,#2\right)}}
\newcommand\G{\mathbb{G}}
\newcommand\R{\mathbb{R}}
\newcommand\Z{\mathbb{Z}}
\newcommand\cA{\mathcal{A}}
\newcommand\cE{\mathcal{E}}
\newcommand\cI{{\mathcal I}}
\newcommand\cJ{{\mathcal J}}
\newcommand\cM{\mathcal{M}}
\newcommand\cR{{\mathcal R}}
\newcommand\cS{{\mathcal S}}
\newcommand\cT{{\mathcal T}}
\newcommand{\fS}{{\mathfrak S}}

\newcommand{\e}{{\mathrm{e}}} 
\newcommand{\dd}{{\mathbf d}}
\newcommand{\xx}{{\mathbf x}}
\newcommand{\yy}{{\mathbf y}}
\newcommand{\tbin}[2]{{\textstyle\binom{#1}{#2}}}
\newcommand{\pd}[2]{\frac{\partial #1}{\partial #2}}

\newcommand{\By}[2]{\overset{\mbox{\tiny{#1}}}{#2}}
\newcommand{\ByRef}[2]{   \By{\eqref{#1}}{#2} }

\newcommand{\geBy}[1]{    \By{#1}{\ge} }
\newcommand{\eqByRef}[1]{ \ByRef{#1}{=} }
\newcommand{\justify}[1]{\fbox{\scriptsize{#1}}\quad}

\newcommand{\citer}[2]{\cite{#1}*{#2}}

\begin{document}
\date{}
\title{On the upper tail of star counts in random graphs}
\author{Margarita Akhmejanova}
\author{Matas \v{S}ileikis}
\address{Institute of Computer Science of the Czech Academy of Sciences, Pod Vod\'{a}renskou v\v{e}\v{z}\'{\i} 2, 182 00 Prague, Czechia. E-mail: {\tt matas@cs.cas.cz}.}
\thanks{M\v{S} was supported by the long-term strategic development financing of the Institute of Computer Science (RVO: 67985807)}

\begin{abstract}
   Let $X$ count the number of $r$-stars in the random binomial graph $\G(n,p)$. We determine, for fixed $r$ and $\eps > 0$, the asymptotics of $\log \prob{X \ge (1 + \eps)\E X}$ assuming only $\E X \to \infty$ and $p \to 0$ thus giving a first class of irregular graphs for which the \emph{upper tail problem} for subgraph counts (stated by Janson and Ruci\'nski in 2004) is solved in the sparse setting.
\end{abstract}
\maketitle
  \section{Introduction}
  Given an integer $r \ge 2$, let $X = X^{K_{1,r}}_{n,p}$ be the number of $r$-armed stars in the random binomial graph $\G(n,p)$ and write $\mu := \E X = n \binom{n-1}{r}p^r$. We study the asymptotics, as $n \to \infty$, of
  \[
    -\log \prob{X \ge (1 + \eps) \mu}
  \]
  for an arbitrary constant $\eps > 0$. 
  This is an instance of the \emph{upper tail problem} that asks the same question for the number $X = X^H_{n,p}$ of copies of an arbitrary (fixed) graph $H$, see Janson and Ruci\'nski \citer{JRdeletion}{Problem 6.4}. For a summary of the exciting history of this intensively studied and elusive problem, see, e.g., \cite{HMS22} and the references therein.

  We are mainly interested in the range of $p = p(n)$ which corresponds to a natural condition $\mu \to \infty$ (if $\mu \to \lambda$, then $X$ converges to Poisson($\lambda$), so there is no exponential decay), and we will assume $p \to 0$ (even though the case $p = \Theta(1)$ remains an interesting analytic problem in view of the large deviation principle established by Chatterjee and Varadhan~\cite{CV11}). 

  \v{S}ileikis and Warnke \cite{SW20} proved that for any constant $\eps > 0$, if $p \to 0$ is such that $\mu \to \infty$, then 
  \begin{equation}
    \label{eq:stars_UT_SW}
    - \log \prob {X \ge (1 + \eps) \mu} \asymp \min \left\{ \phi(\eps)\mu, \max \left\{ (\eps \mu)^{1/r}, \eps n^2 p^r \right\} \log (1/p) \right\} \ ,
  \end{equation}
  where $a_n \asymp b_n$ stands for $a_n = \Theta(b_n)$ and $\phi$ is a strictly increasing function, defined as
  \begin{equation}
    \label{eq:def_phi}
    \phi(\eps) := (1+\eps)\log(1+\eps) - \eps, \qquad \eps \ge 0 \ .
  \end{equation}
  (Henceforth $\log$ stands for the natural logarithm.)

  Ignoring the dependency on $\eps$, the right-hand side of \eqref{eq:stars_UT_SW} is of order $\Phi_n$, where
  \begin{equation}
    \label{eq:Phi_order}
    \Phi_n :=
    \begin{cases}
      n^{r+1}p^r \quad &\text{if }p \le n^{-1 - 1/r} \log^{\frac{1}{r-1}} n \\
      n^{1 + 1/r} p \log n \quad &\text{if }n^{-1 - 1/r} \log^{\frac{1}{r-1}} n < p \le n^{-1/r} \\
      n^2 p^r \log \frac{1}{p} \quad &\text{if }p > n^{-1/r}.
    \end{cases}
  \end{equation}

  Formula \eqref{eq:stars_UT_SW} gives the asymptotics up to an absolute constant. Let us recall what is known about precise asymptotics.
  Harel, Mousset and Samotij \cite{HMS22} mention that their Theorem~1.6 can be straightforwardly extended (which we have verified) to so-called strictly balanced graphs $H$. Since $K_{1,r}$ is strictly balanced, this implies that 
  \begin{equation}
\label{eq:Poisson_regime}
-\log \prob{X \ge (1 + \eps) \mu} \sim \phi(\eps) \mu \quad \text{whenever} \quad 1 \ll \mu \ll (\log n)^{r/(r-1)} \ .
  \end{equation}
  (See \cite{SWcounterexample} for examples that illustrate complications with $H$ that are not strictly balanced.)
  
  From a result by Cook, Dembo and Pham (see \citer{CDP24}{Corollary 3.3(c)}) it follows that
  \begin{equation}
    \label{eq:CDP}   
    -\log \prob{X \ge (1 + \eps) \mu} \sim \eps n^2 p^r \log (1/p) \quad \text{whenever} \quad n^{-1/r} \ll p \ll 1 \ .
  \end{equation}

  In this paper, we cover the missing range of $p$ for stars $H = K_{1,r}$ given by $\mu = \Omega(\log^{r/(r-1)} n)$ and $p = O(n^{-1/r})$ (the latter condition being equivalent to $\mu = O(n^r)$). 
  Our result goes beyond the regular graphs covered by the paper by Basak and Basu~\cite{BB23} who extended the breakthrough paper of Harel, Mousset and Samotij \cite{HMS22}. As a bonus, our result covers the transition between Poisson and localization regimes (defined below). Authors of \cite{HMS22} expressed a belief that an analogous transition holds for regular $H$, compare~\citer{HMS22}{equation (68)} and \eqref{eq:unified}. 

  At the Polish Combinatorial Conference in September 2024, we became aware that Antonir Cohen, Harel, Mousset and Samotij \cite{AHMS} are preparing a paper on the upper tail for a class of irregular graphs $H$ that includes stars. They describe the asymptotics of the log-probability implicitly as a variational problem (which was solved earlier in \cite{BGLZ} in the case $p \gg n^{-1/\Delta_H}$ that corresponds to $\mu \gg n^r$ for $H = K_{1,r}$). For $H = K_{1,r}$ their result covers the last three cases of \eqref{eq:stars_UT}.

  We are now ready to state our result. Condition $\mu \gg \log n$ is an artifact of the way we approximate the degree sequence of $\G(n,p)$ by i.i.d.\ binomial random variables and can be relaxed to $\mu \gg 1$ by a more careful argument. The combination of \eqref{eq:Poisson_regime} and Theorem~\ref{thm:stars_UT} implies that \eqref{eq:stars_UT} holds for every $\mu \gg 1$. Below $\{x\} := x - \floor{x}$ stands for the fractional part of $x$.
  \begin{theorem}
    \label{thm:stars_UT}
    Fix an integer $r \ge 2$ and let $X = X^{K_{1,r}}_{n,p}$ be the number of copies of star $K_{1,r}$ in the random graph $\G(n,p)$.
    For $\delta \ge 0$ define a function $\psi_r(\delta) = r^{-1}(r!\delta)^{1/r}$ and let
    \begin{equation}
      \label{eq:I_def}
      I_r(c,\eps) = \min_{\delta \in [0,\eps]} \left( \phi(\eps - \delta) + \psi_r(\delta)c^{1/r - 1} \right), \quad c, \eps > 0 \ .
  \end{equation}
      Furthermore, set $I_r(0, \eps) = \phi(\eps)$.
  If $p = p(n) \to 0$ is such that $\mu \gg \log n$, then
    \begin{equation}
      \label{eq:stars_UT}
      -\log \prob{X \ge (1 + \eps)\mu} \sim 
       \begin{cases}
	 I_r(c,\eps) \mu  \quad &\text{if } \mu/\log^{r/(r-1)} n \to c \in [0, \infty) \\
	 \psi_r(\eps)\mu^{1/r} \log n\quad &\text{if } \log^{r/(r-1)} n \ll \mu \ll n^r \\
        \frac{1}{r} \left( \left\{ \eps \rho \right\} ^{1/r} + \floor{\eps \rho}  \right) n \log n  \quad &\text{if } \mu/ \binom{n}{r} \to \rho \in (0, \infty) \\
	\eps n^2 p^r \log \frac{1}{p} \quad &\text{if } \mu \gg n^r  \ .
      \end{cases}  
    \end{equation}
  \end{theorem}
  The authors of \cite{HMS22} call the range of $p$ where \eqref{eq:Poisson_regime} holds the \emph{Poisson regime}. It is well known that the asymptotics \eqref{eq:Poisson_regime} hold if $X$ has a Poisson distribution with expectation $\mu \to \infty$. 
  \begin{remark}
    \label{rem:minimizer}
    For a description of the minimizers in \eqref{eq:I_def} see Proposition~\ref{prop:minimizers}. In particular, it implies that the Poisson regime holds if and only if $\limsup_{n \to \infty} \mu/\log^{r/(r-1)} n \le c$ for some constant $c = c_{r,\eps}$.
  \end{remark}
  If one desires a unified expression for the asymptotics, one could show, using Proposition~\ref{prop:minimizers} and the subsubsequence principle, that for fixed $\eps$ and any $n^{-1-1/r} \ll p \ll 1$
  \begin{equation}
\label{eq:unified}
    -\log \prob{X \ge (1 + \eps)\mu} \sim \min_{\delta \in [0, \eps]} \left[ \phi(\eps - \delta)\mu + \Psi(\delta) \right],
  \end{equation}
  where
  \begin{equation*}
    \Psi(\delta) := r^{-1}\left\{ {\delta\mu}/{\tbin{n}{r}} \right\} ^{1/r} n \log n + n \floor{{\delta \mu}/{\tbin{n}{r}}} \log (1/p) \ .
  \end{equation*}
  
  The random variable $X$ is a statistic of the degree sequence of $\G(n,p)$ (note that $r!X$ is an empirical factorial moment). For $p \asymp 1/n$, a large deviation principle with speed $n$ for the empirical degree distribution was obtained by Bordenave and Caputo \citer{BC}{Corollary 1.10} (see also \cite{DAM}) who mention that interesting large deviation events occur at other speeds. Theorem~\ref{thm:stars_UT} gives speed $n^{1/r}\log n$ for the tail event of $r$-stars. Our method extends to statistics $\sum_i f(X_i)$ for more general convex functions $f$; we encourage the readers to adapt the proof to their needs.
  
  \subsection{Localisation: comparison with regular graph counts}  Poisson regime is contrasted with the \emph{localisation regime}, in which large deviations mainly occur due to appearance of a relatively small and simple structure. The standard approach to prove a lower bound for the upper tail is to choose a graph $G \subset K_n$ such that \begin{equation}
    \label{eq:seed}
    \Ec{X}{G \subset \G(n,p)} \ge (1 + \eps)\mu. \end{equation}
  If one conditions on the event $G \subset \G(n,p)$ (which has probability $p^{e(G)}$), the upper tail event becomes a typical event, making it plausible that
  \begin{equation}
\label{eq:lower_loc}
    \log \prob{X \ge (1 + \eps)\mu} \ge (1 + o(1))\log p^{e(G)} \sim -e(G) \log (1/p) \ .
  \end{equation}
  The best such bound is obtained by minimizing the number of edges over graphs satisfying \eqref{eq:seed}.
  As proved in \cite{HMS22} (and extended in \cite{BB23}), for connected \emph{regular} $H$ this lower bound can be matched by the upper bound (which is the hard part of the problem). 

  These heuristics still work for \emph{stars} in the regime $p \gg n^{-1/r}$ (cf. \eqref{eq:CDP}), since then the expectation of the star count is boosted by $\eps \mu$ by the complement of a clique on ${n-a}$ vertices, for some $a \sim (\eps np^r)$, which has asymptotically $\eps n^2p^r$ edges. However, for smaller $p$ it matters not only how many edges an optimal graph $G$ has but also how many such graphs $K_n$ contains. For example, considering $p = n^{-c}$ with constant $c \in (1/r, 1 + 1/r)$, one could show that the optimal graph $G$ is a star with $m \sim \left( r! \eps \mu \right)^{1/r}$ edges, so the bound \eqref{eq:lower_loc} is asymptotically $-m \log (1/p) = -c m \log n$. However the constant $c$ can be improved to $1/r$ by considering the event that a particular vertex has degree $m$, the logarithm of which can be shown (see \eqref{eq:weak_Chernoff} and \eqref{eq:binom_lower}) to satisfy
  \begin{equation*}
    \log \left[ \binom{n-1}{m} p^m (1 - p)^{n-1-m} \right] \sim - m \log \frac{m}{np} \sim - m \log n^{1/r} = - \tfrac{1}{r} \cdot m \log n \ .
  \end{equation*}

  \subsection{Plan of the proof}
  The general idea is to treat the number of stars as a functional of the degree sequence. The degree sequence of $\G(n,p)$ is in some sense similar to a vector of i.i.d.\ binomial random variables. Using asymptotic enumeration of graphs with a given degree sequence due to McKay~\cite{M85}, in Section~\ref{s:reduction} we prove a couple of lemmas which allow us to focus on i.i.d.\ vectors instead of the degree sequence; this approach works in some sparse regime that contains the whole transition between the Poisson and localization regimes. In Section~\ref{s:iid} we determine the asymptotics of the log-probability for i.i.d.\ vectors, treating separately the contribution of small values (which give the bulk of the sum and are main contributors of excess deviation in the Poisson regime) and large values (which are main contributors of excess deviation in the localized regime). The enumeration approach is not feasible in the denser regime, so we work directly with degrees of $\G(n,p)$. Luckily, this range is deeply within the localized regime, which allows us to bound the contribution of the small degrees in a cruder way using an inequality due to Warnke (in a somewhat similar way as it was applied in \cite{SW20}). The treatment of the contribution of large degrees is relatively bare-handed and can be seen as a refinement of \cite{SW20}, inspired by \citer{AGSW23}{Lemma~21}.

  \section{Preliminaries}
  We use asymptotic notation $O, \Omega, \Theta, o, \omega, \ll, \gg, \asymp, \sim$, always with respect to the variable $n$ tending to infinity unless stated otherwise.
  \subsection{Tail bounds}
  Given $p \in (0,1)$, define a function $H_p : [0, 1] \to [0, \infty)$ by setting, for $p \in (0,1)$,
  \begin{equation}
    \label{eq:Hp_def}
    H_p(\lambda) := \lambda \log \frac{\lambda}{p} + (1-\lambda) \log \frac{1-\lambda}{1-p}
  \end{equation}
  and extending it to $p = 0, 1$ by taking limits.
  As is well known, function $H_p$ is strictly convex and attains its minimal value of $0$ at $\lambda = p$.
  Chernoff's inequality (see, e.g., \citer{JLRbook}{(2.4)}) asserts that if $X \sim \Bin(n,p)$, then
  \[
    \prob{X \ge \lambda n} \le \e^{-nH_p(\lambda)}, \qquad \lambda \in [p, 1].
  \]
  We will often use the following weakening of Chernoff's inequality (see, e.g., \citer{AGSW23}{Eq. (22)}), which is suitable for deviations much larger than the expectation (in fact it is useless for $t \le \e np$):
  \begin{equation}
    \label{eq:weak_Chernoff}
    \prob{X \ge t} \le \exp \left( - t \log \frac{t}{\e np} \right), \quad t > 0.
  \end{equation}
  The following lower bound gives the same asymptotics for the log-probability when $t \gg np$.
  \begin{lemma}
    Let $X \sim \Bin(n,p)$. If integer $k$ satisfies $np \ll k \le n$, then
    \begin{equation}
    \label{eq:binom_lower}
      \prob{X \ge k} \ge \exp \left( - (1+ o(1))k \log \frac{k}{np}  \right) \ .
    \end{equation}
  \end{lemma}
  \begin{proof}
    If $k \le n/2$, then
    \[
      \prob{X \ge k} \ge \binom{n}{k}p^{k}(1-p)^{n - k} \ge \left( \frac{n/2}{k} \right)^k p^k \e^{-O(np)} = \e^{ -k \log \frac{2k}{np} + O(np) } = \e^{  - (1 + o(1))k \log \frac{k}{np} } \ .
    \]
    For $k > n/2$ note that if the first $k$ underlying Bernoulli random variables take value $1$, then $X \ge k$ and therefore
    \[
      \prob{X \ge k} \ge p^{k} = \e^{ -k \log \frac{1}{p}}  = \e^{ -k \left( \log \frac{k}{np} + \log \frac{n}{k}\right) } = \e^{  - (1 + o(1))k \log \frac{k}{np} } \ .
    \]
  \end{proof}
  
  We will also use the following inequality that is a corollary of a result by Warnke \cite{W17}.
  \begin{theorem}[{\citer{SW20}{Theorem~7}}]
    \label{thm:C}
    Let $(\xi_i)_{i \in \fS}$ be a finite family of independent random variables with~$\xi_i \in \{0,1\}$. 
    Given a family $\cI$ of subsets of $\fS$, consider random variables $Y_{\alpha} := \prod_{i \in \alpha}\xi_i$ with~$\alpha \in \cI$, and suppose $\sum_{\alpha \in \cI} \E Y_{\alpha} \le \mu$. 
    Define $Z_C :=\max \sum_{\alpha \in \cJ} Y_{\alpha}$, where the maximum is taken over 
    all $\cJ \subseteq \cI$ with 
    \begin{equation}
      \label{eq:max_cond}
      \max_{\beta \in \cJ}|\{\alpha \in \cJ: \alpha \cap \beta \neq \emptyset\}| \leq C. 
    \end{equation}
    Then, for all $C,t>0$, 
    \begin{equation*}
      \prob{Z_C \geq \mu +t }
      \le \e^{-{\phi(t/\mu)\mu}/{C} }  .
    \end{equation*} 
  \end{theorem}

\subsection{Asymptotic enumeration of graphs with a given sequence}
Given a degree sequence $\dd = (d_1, \dots, d_n)$, let $m(\dd) := \frac{1}{2} \sum_i d_i$ be the number of edges and $\lambda = \lambda(\dd) := {m(\dd)}/{\binom{n}{2}}$ be the edge density in a graph with degree sequence $\dd$. A formula by McKay and Wormald \citer{MW91}{Theorem 1.1} (which can be read out of an older paper by McKay~\cite{M85}) asserts that if
\begin{equation}
  \label{eq:enum_cond} \max_i d_i \ll m(\dd)^{1/4} ,
\end{equation}
then the number \( g(\dd) \) of graphs with degree sequence \( \dd \) satisfies 
  \begin{equation}
    \label{eq:enum}
    { g(\dd)}  \sim \sqrt{2} \exp \left( \frac{1}{4} - \left( \frac{\gamma_n}{2\lambda(1-\lambda)}  \right)^2 \right) \left( \lambda^\lambda (1 - \lambda)^{1 - \lambda} \right)^{\binom{n}{2}} \prod_{i \in [n]} \binom{n-1}{d_i}, 
  \end{equation}
  where, denoting the average degree by $d$,
  \begin{equation}
    \label{eq:gamma_n_def}
      \gamma_n = \gamma_n(\dd) := (n-1)^{-2} \left(\sum_{i = 1}^n d_i^2 - nd^2 \right).
  \end{equation}
  Often in the formula \eqref{eq:enum} the exponent involving $\gamma_n$ is shown to be $o(1)$ or at least $O(1)$.

\subsection{A basic property of convex functions}
\begin{proposition}
  \label{prop:convex_sum}
  Let $f : \{0, \dots, N\} \to \R$ be an increasing convex function satisfying $f(0) = 0$ and $f(N) > 0$. Let $t \in [0, nf(N)]$. 
  For any integers $m_1, \dots, m_n \in [0,N]$ we have that \begin{equation}
      \label{eq:sum_lower}
      \sum_{i=1}^n f(m_i) \ge t \quad \text{implies} \quad \sum_{i=1}^n m_i \ge \floor{{t}/{f(N)}} N + b  \ ,
  \end{equation}
  where $b := \min \{ m : f(m) \ge \left\{t/f(N)\right\}f(N) \}$.
\end{proposition}
\begin{proof}
  Our goal is to minimize $\sum_i m_i$ subject to $\sum_i f(m_i) \ge t$.
  First note that by convexity any integers $1 \le a \le b \le N - 1$ satisfy
  \begin{equation}
\label{eq:basic}
    f(a) + f(b) \le f(a-1) + f(b+1).
  \end{equation}
  Therefore we can assume that the vector $(m_1, \dots, m_n)$ has at most one entry that is not in $\left\{ 0,N \right\}$, since otherwise, picking indices $i,j$ with $1 \le m_i < m_j \le N - 1$ and replacing $m_i$ by $m_i - 1$ and $m_j$ by $m_j+1$, in view of \eqref{eq:basic} maintains the constraint $\sum_i f(m_i)$ without changing the sum. 
  Among the vectors with at most one entry not in $\{0,N\}$, the minimum sum is clearly attained by vectors with $\floor{t/f(N)}$ entries $N$, one entry $b$ and remaining entries zero, implying \eqref{eq:sum_lower}.
\end{proof}

  \section{Reduction to i.i.d.~binomials}
  \label{s:reduction}
  Using the enumeration formula \eqref{eq:enum} we will show that in a certain sparse regime one can reduce the study of the upper tail of stars to the upper tail of the sum
  \[
  Y := \sum_{i=1}^n \binom{X_i}{r},
  \]
  where $X_1, \dots, X_n$ are i.i.d.\ random variables with distribution $\Bin(n-1,p)$.
  For this we define two probability measures on the set of vectors $\{0, \dots, n-1\}^n$: let $\prb{p}{}$ be the distribution of $(X_1, \dots, X_n)$ and let $\prd{p}{}$ be the distribution of the degree sequence of a random graph $\G(n,p)$.

  Write $\Phi_n$ as in \eqref{eq:Phi_order}. For any positive $C > 0$, let
  \begin{equation}
\label{eq:def_deltas}
    \delta_n = \delta_{C,n} := \sqrt{\frac{2C \Phi_n}{ \binom{n}{2}p}}, \qquad \Delta_n = \Delta_{C,n} := \frac{C(\log n +  \Phi_n)}{\log \frac{\log n + \Phi_n}{np}} \ .
  \end{equation}
  Writing $\dd = (d_1, \dots, d_n)$ and $m(\dd) = \frac{1}{2}\sum_{i \in [n]} d_i$, for $C > 0$ define a set of `well-behaved' vectors 
  \begin{equation}
    \label{eq:def_cR}
    \cR_{C, n} := \left\{ \dd \in \{0, \dots, n-1\}^n : \max_{i \in [n]} d_i \le \Delta_n, \left|\frac{m(\dd)}{\binom{n}{2}p}  - 1\right| < \delta_n \right\} \ .
  \end{equation}
  \begin{lemma}
    \label{lem:goodness}
    Suppose that $n^{-1-1/r}\ll p \le n^{-1/r}$.  
    For every constant $C>1$, we have
    \begin{equation}
      \label{eq:bad_sequences}
      \prdp{\overline{\cR_{C,n}}} \le \e^{-(C + o(1))\Phi_n} \quad\text{and}\quad \prbp{\overline{\cR_{C,n}}} \le \e^{-(C + o(1))\Phi_n}.
    \end{equation}
  \end{lemma}
  \begin{proof}
    Definition \eqref{eq:Phi_order} of $\Phi_n$ gives that
    \begin{equation}
      \label{eq:Phi_np}
      \frac{\Phi_n}{np} \asymp \left.
    \begin{cases}
      n^{r}p^{r-1} \quad &\text{if } n^{-1-1/r} \ll p \le n^{-1 - 1/r} \log^{\frac{1}{r-1}} n \\
      n^{1/r}\log n \quad &\text{if } n^{-1 - 1/r} \log^{\frac{1}{r-1}} n < p \le n^{-1/r} .
    \end{cases} \right\} \gg 1 \ .
    \end{equation}
    Therefore $\frac{\log n + \Phi_n}{np} \to \infty$. For $X \sim \Bin(n-1,p)$, Chernoff's inequality \eqref{eq:weak_Chernoff} implies
    \begin{align*}
      -\log \prob{X > \Delta_n} &\ge  \Delta_n \log \frac{\Delta_n}{\e (n-1)p} = \Delta_n \left( \log \frac{\log n + \Phi_n}{np} - \log \log \frac{\log n + \Phi_n}{np} + \log \frac{Cn}{\e(n-1)} \right) \\
\justify{\eqref{eq:Phi_np}}      &\sim  \Delta_n \log \frac{\log n + \Phi_n}{np} = C(\log n + \Phi_n).
    \end{align*}
    As $C > 1$, we obtain $\prob{X > \Delta_n} \le \frac{1}{n}\e^{-(C+o(1))\Phi_n}$.
    Applying the union bound over $n$ events, we get that the probability of $\max_{i \in [n]} d_i > \Delta_n$ is at most $\e^{-(C + o(1))\Phi_n}$ for both $\prd{p}{}$ and $\prb{p}{}$.

    Turning to the deviations of $m(\dd)$, note that
    \begin{equation}
      \label{eq:delta_small}
      \delta_n^2 = \frac{2C\Phi_n}{\binom{n}{2}p} \asymp
      \left.\begin{cases}
      n^{r-1}p^{r-1} \quad &\text{if }p \le n^{-1 - 1/r} \log^{\frac{1}{r-1}} n \\
      n^{1/r-1}\log n \quad &\text{if }n^{-1 - 1/r} \log^{\frac{1}{r-1}} n < p \le n^{-1/r}  
  \end{cases} \right\}\ll 1 \ .
    \end{equation}
    Since $m(\dd) \sim \Bin(\binom{n}{2},p)$ under $\prd{p}{}$, a version of Chernoff's bound (say \citer{JLRbook}{Theorem~2.1}) gives
    \begin{align*}
      \prdp{\left|m(\dd) - \tbin{n}{2}p\right| \ge \delta_n \tbin{n}{2} p } \le  2 \exp \left(  - \frac{\delta_n^2 \binom{n}{2}p}{2 + \delta_n/3} \right)  = 2 \e^{- \tfrac{1}{2}\delta_n^2 \tbin{n}{2}p (1 + o(1))} = \e^{ -(C + o(1))\Phi_n} 
    \end{align*}
    and since $2m(\dd) \sim \Bin((n)_2,p)$ under $\prb{p}{}$,
    \begin{align*}
      \prbp{\left|2m(\dd) - (n)_2p\right| \ge \delta_n (n)_2 p } \le  2 \exp \left(  - \frac{\delta_n^2 (n)_2p}{2 + \delta_n/3} \right)  = 2 \e^{- \tfrac{1}{2}\delta_n^2 (n)_2p (1 + o(1))} = \e^{ -(2C + o(1))\Phi_n} \ .
    \end{align*}
    This, together with what we have shown about the maximum degree, proves \eqref{eq:bad_sequences}.
  \end{proof}
 
  In the following lemma, the condition $p \ll n^{-2/3 - 4/(3r)}$ stems from the condition \eqref{eq:enum_cond}. In this paper, we will apply Lemma~\ref{lem:reduction} under a stronger assumption $p \ll n^{o(1) - 1 - 1/r}$ (for $r\ge2$).
  \begin{lemma}
    \label{lem:reduction}
    Fix an integer $r \ge 2$ and $C > 0$. Let $\cR_n = \cR_{C,n}$ and consider an arbitrary set $\cS_n \subseteq \left\{ 0, \dots, n-1 \right\}^n$. If $n^{-1-1/r} \ll p \ll n^{-2/3 - 4/(3r)}$ then with $\delta_n = \delta_{C,n}$ defined in \eqref{eq:def_deltas},
    \begin{equation}
      \label{eq:reduct_upper}
      \prdp{\cR_n \cap \cS_n} \le O\left(\delta_n \tbin{n}{2}p\right)\max_{p' \in [(1-\delta_n)p, (1 + \delta_n)p]} \prb{p'}{\left(\cS_n\right)}.
    \end{equation}

  Moreover, defining $\cT_n :=\{\dd \in \{0, \dots, n-1\}^n : \sum_i d_i^2 \le 2n^2p\}$, we have
    \begin{equation}
      \label{eq:reduct_lower}
      \prdp{\cR_n \cap \cS_n \cap \cT_n} \ge \Omega(1) \cdot \left(   \prbp{\cS_n} - \prbp{\overline{\cR_n} \cup \overline{\cT_n}} \right) .
    \end{equation}
  \end{lemma}

  \begin{proof}
    First we show that in the assumed range of \( p \) condition \eqref{eq:enum_cond} is satisfied by every  \( \dd \in \cR_n \) and therefore the graph enumeration formula \eqref{eq:enum} applies. 
    
    Recall the definition of $\Delta_n = \Delta_{C,n}$ from \eqref{eq:def_deltas}. Since $\delta_n \to 0$ (see \eqref{eq:delta_small}), $\dd \in \cR_n$ implies $m(\dd) \asymp n^2p$, and thus it suffices to check that \(\Delta_n \ll (n^2 p)^{1/4} \). We consider cases (i) $n^{-1-1/r} \le p \le n^{-1-1/r}\log^{1/(r-1)}n$ and (ii) $p > n^{-1-1/r}\log^{1/(r-1)}n$. In the case (i) we have \(\Phi_n = n^{r+1}p^r = O(\log^{r/(r-1)}n) \) and \( \log n + \Phi_n \gg np\). So the denominator of $\Delta_n$ tends to infinity and we obtain that 
    \[
      \Delta_n \ll \log n+\Phi_n = O(\log^{r/(r-1)} n) \ll  (n^2 p)^{1/4} \ .
    \] 
    In the case (ii), \( \Phi_n = n^{1+1/r} p \log n \gg \log n\), which implies \(\log ((\Phi_n+ \log n)/(np))\asymp \log n\). Hence 
    \[
      \Delta_n \asymp \frac{n^{1+1/r}p \log n}{\log n} = n^{1+1/r}p^r,
    \]
    which can be seen to be $o( (n^2p)^{1/4})$ due to $p \ll n^{-\frac{2}{3} - \frac{4}{3r}}$.

    Given $\dd \in \cR_n$, write $m = \frac{1}{2}\sum_i d_i$ and $\lambda = {m}/{{\binom{n}{2}}}$. Since in $\G(n,p)$ the probability of every graph with degrees $\dd$ is $p^m(1-p)^{\binom{n}{2}-m}$, we have $\prdp{\dd} = g(\dd) p^m(1-p)^{\binom{n}{2}-m}$. Hence, using formula \eqref{eq:enum}, we obtain 
    \begin{align}
    \notag \frac{\prdp{\dd}}{\prb{\lambda}{\left(\dd\right)}} 
      &= \frac{g(\dd)p^m(1-p)^{{\binom{n}{2}} - m} }{ \prod_{i=1}^n\binom{n-1}{d_i}\lambda^{d_i}(1-\lambda)^{n-1-d_i}} \\
      \label{eq:DpBlambda_ratio} &= \frac{g(\dd)p^m(1-p)^{{\binom{n}{2}} - m}}{\lambda^{2m}(1-\lambda)^{2{\binom{n}{2}}-2m}\prod_{i=1}^n\binom{n-1}{d_i}} \\
      \justify{\eqref{eq:Hp_def}, $H_p \ge 0$} \notag &\le O(1) \cdot \left( \left(\frac{p}{\lambda}  \right)^\lambda \left( \frac{1-p}{1-\lambda} \right)^{1-\lambda} \right)^{\binom{n}{2}} = O\left( \e^{-{\binom{n}{2}} H_p(\lambda)} \right) = O(1).
    \end{align}
    Let $M_n := \left\{ m \in \Z : \left|\frac{m}{\binom{n}{2}p} - 1\right| \le \delta_n \right\}$.
    Since $\dd \in \cR_n$ implies $m(\dd) \in M_n$, we obtain
    \begin{equation*}
      \prdp{\cR_n \cap \cS_n} 
      \le O(1) \cdot \sum_{m \in M_n} \sum_{\dd \in \cR_n \cap \cS_n : m(\dd) = m}\prb{m/{\binom{n}{2}}}{\left(\dd\right)} \le O(1)\sum_{m \in M_n  }\prb{m/{\binom{n}{2}}}{\left(\cS_n\right)} \ .
    \end{equation*}
    Since $|M_n| = O(\delta_n {\binom{n}{2}} p)$ and for $m \in M_n$ we have $m/\binom{n}{2} = (1 \pm \delta)p$, inequality \eqref{eq:reduct_upper} follows.

    To prove inequality \eqref{eq:reduct_lower}, we first bound the exponent involving $\gamma_n$ in \eqref{eq:enum} for sequences $\dd \in \cR_n \cap \cT_n$. Since $\dd \in \cT_n$ we have 
    \[
      \gamma_n(\dd) \le (n-1)^{-2}\sum_{i = 1}^n d_i^2 \le (2+ o(1)) p.
    \]
    Further, since $\dd$ satisfies the second inequality in \eqref{eq:def_cR}, and $\delta_n \to 0$ (see \eqref{eq:delta_small}) we have $\lambda(\dd) \sim p \to 0$, whence 
    \begin{equation}
      \label{eq:gamma_bound}
     \frac{\gamma_n(\dd)}{2\lambda(1-\lambda)} \le (1 + o(1)) \frac{2p}{2p(1-o(1))} \sim 1.
     \end{equation}
    Repeating the steps leading to equality \eqref{eq:DpBlambda_ratio} we obtain (note that we use $\mathrm{B}_p$ instead of $\mathrm{B}_\lambda$) 
    \begin{align*}
      \frac{\prdp{\dd}}{\prbp{\dd}} 
      &= \frac{g(\dd)p^m(1-p)^{{\binom{n}{2}} - m}}{p^{2m}(1-p)^{2{\binom{n}{2}}-2m}\prod_{i=1}^n\binom{n-1}{d_i}} = \frac{g(\dd)}{p^{m}(1-p)^{{\binom{n}{2}}-m}\prod_{i=1}^n\binom{n-1}{d_i}}\\
      \justify{\eqref{eq:enum}} &\sim \sqrt{2} \exp \left(\tfrac{1}{4}-\left(\frac{\gamma_n}{2\lambda(1-\lambda)}\right)^2 \right)  \left( \left(\frac{\lambda}{p}  \right)^\lambda \left( \frac{1-\lambda}{1-p} \right)^{1-\lambda} \right)^{\binom{n}{2}} \\
      \justify{\eqref{eq:gamma_bound}, $H_p \ge 0$}   &\asymp \e^{{\binom{n}{2}} H_p(\lambda)}  = \Omega(1).
    \end{align*}
    Consequently $\prdp{\cR_n \cap \cS_n \cap \cT_n} = \Omega\left( \prbp{\cR_n \cap \cS_n \cap \cT_n} \right)$ from which \eqref{eq:reduct_lower} easily follows.
\end{proof}

  \section{Upper tail for i.i.d.\ binomials}
  \label{s:iid}
  In this section, fixing $r \ge 2$ and given a sequence of integers $N = N(n)$, we consider the random variable 
  \[ 
    Y = \tbin{X_1}{r} + \dots + \tbin{X_n}{r},
  \]
  where $X_i$ are independent random variables with distribution $\Bin(N,p)$. 
  Let $\nu = \E Y = n \binom{N}{r}p^r$.
  Focusing on the setting when $\nu \to \infty$ but $Np \to 0$, our goal is to obtain  the asymptotics of
  \[
    \log \prob {Y \ge (1 + \eps) \nu}.
  \]
  Moreover, in a couple of situations the proof will easily extend to the random variable $X$ counting $r$-stars in $\G(n,p)$, giving the same log-probability asymptotics for the sequence $N(n) = n-1$.

  Our approach is to split the terms of $Y$ according to some cutoff value $R = R(n)$. We define
    \begin{equation}
\label{eq:Y_primes}
      Y' = \sum_{i \in [n] : X_i \le R} \binom{X_i}{r}, \qquad  Y'' = \sum_{i \in [n] : X_i > R} \binom{X_i}{r},
    \end{equation}
    so that $Y = Y' + Y''$.

  \subsection{Small values}
  First we deal with the contribution of the small values which are expected to be the leading contributor in the Poisson regime.
\begin{lemma}
  \label{lem:small_value_sum_LDP}
  Suppose that $Np \to 0$ and $\nu \to \infty$. 
  If the sequence $R = R(n)$ satisfies
  \begin{align}
\label{eq:R_def}
r \le R \ll \left( \log \frac{1}{Np} \right)^{1/(r-1)} ,
  \end{align}
  then for any constant $\eps > 0$ the random variable $Y'$ defined in \eqref{eq:Y_primes} satisfies
    \begin{equation*}
      \log \prob{Y' \ge (1 + \eps)\nu} \sim -\phi(\eps)\nu \ .
    \end{equation*}
\end{lemma}
\begin{proof}
  Observe that $Y' = \sum_{i=1}^n \binom{Z_i}{r}$, 
where $Z_i := X_i \ind{\{X_i \le R\}}$ are i.i.d.\ and have distribution obtained from $\Bin(N,p)$ by moving the mass from values greater than $R$ to value $0$. In other words, defining $\zeta_j := \prob{Z_1 = j}$ we have
\begin{equation*}
  \zeta_0 = \prob{X_1 = 0} + \prob{X_1 > R}, \quad \text{and }  \quad \zeta_j = \prob{X_1 = j}, \ j= 1, \dots, R \ .
\end{equation*}
Starting with the upper bound, let $t := (1 + \eps)\binom{N}{r}p^r$. For every $h > 0$, using independence we get
 \begin{align}
   \notag \prob{Y' \ge (1 + \eps)\nu} &= \prob{Y' \ge nt} \le \E \exp({h(Y' - nt)}) = \left( \E \exp \left( h\left[\textstyle{\binom{Z_1}{r}} - t \right]\right) \right)^n \\
   \label{eq:exp_mom_upper} &= \left( \frac{\sum_i \e^{h\binom{i}{r}} \zeta_i}{\e^{ht}} \right)^n = \exp \left( n (\Lambda(h) - ht)  \right),
 \end{align}
 where $\Lambda(h) := \log \E \e^{h\binom{Z_1}{r}}$. This upper bound could be optimized over $h > 0$, leading to equation
 \[
   t = \Lambda'(h) = \frac{\E \binom{Z_1}{r}\e^{h\binom{Z_1}{r}}}{\E \e^{h\binom{Z_1}{r}}} = \frac{\sum_i \binom{i}{r} \e^{h\binom{i}{r}} \zeta_i}{\sum_i \e^{h\binom{i}{r}} \zeta_i}.
 \]
 It is tricky to find the exact solution of this equation, except for the case $R = r$ which gives $h = h_{\eps} := \log (1 + \eps)$. Luckily, this value of $h$ will turn out to give a sufficiently good bound for any $R$ satisfying \eqref{eq:R_def}. Claiming that
 \begin{align}
   \label{eq:momgen_est} 
   \E \e^{h_\eps \binom{Z_1}{r}}  = \sum_{i = 0}^R (1 + \eps)^{\tbin{i}{r}}\zeta_r \le  1 +\left( \eps + o(1) \right) \tbin{N}{r}p^r ,
 \end{align}
 and using $\log (1 + x) \le x$ we obtain
 \begin{align*}
   \Lambda(h_\eps) -h_\eps t  &\le 
   \log \left( 1 +\left( \eps + o(1) \right) \tbin{N}{r}p^r \right) - \log(1 + \eps) \cdot {(1 + \eps) \tbin{N}{r}p^r }\\
   & \le (\eps + o(1))\tbin{N}{r}p^r - (1 + \eps)\log(1 + \eps)\tbin{N}{r}p^r \sim - \phi(\eps)\tbin{N}{r}p^r,
 \end{align*}
 which we substitute into \eqref{eq:exp_mom_upper} to obtain the desired upper bound. 

 To see why inequality \eqref{eq:momgen_est} holds, note that, for $X \sim \Bin(N,p)$,
 \begin{align}
\notag   \sum_{i = 0}^{r-1} (1 + \eps)^{\tbin{i}{r}} \zeta_i = \sum_{i = 0}^{r-1} \zeta_i = 1 - \prob{X \ge r} + \prob{X \ge R + 1} &\le 1 - \tbin{N}{r}p^r(1-p)^{N-r} + \tbin{N}{R+1}p^{R+1} \\
\label{eq:smaller_upper} \justify{$Np \to 0$}&= 1 - \tbin{N}{r}p^r \left( 1 + o(1) \right)
 \end{align}
We bound the remaining terms in \eqref{eq:momgen_est} by a geometric series after checking that the ratio of consecutive terms tends to zero. For $i = r, \dots, R-1$,
 \begin{align*}
   \notag \frac{(1 + \eps)^{\binom{i+1}{r}}\zeta_{i+1}}{(1 + \eps)^{\binom{i}{r}}\zeta_{i}} &= (1 + \eps)^{\binom{i+1}{r} - \binom{i}{r}} \frac{(N-i)p}{(i+1)(1-p)} \le (1 + \eps)^{i^{r-1}}Np \\
   \label{eq:ratio_small} &\le \exp \left( R^{r-1} \log (1 + \eps) \right)Np \eqByRef{eq:R_def} (Np)^{1 - o(1)} = o(1) 
 \end{align*}
 and therefore
 \[
   \sum_{i = r}^R (1 + \eps)^{\binom{i}{r}}\zeta_{i} \le \frac{(1 + \eps)\zeta_{r}}{1 - o(1)} \sim (1 + \eps) \zeta_r \ .
 \]
 Since $\zeta_r \sim \binom{N}{r}p^r$, together with \eqref{eq:smaller_upper} this implies the inequality in \eqref{eq:momgen_est}.

 For the lower bound we note that
 \begin{equation}  
\label{eq:lower}
   \prob{Y' \ge (1 + \eps)\nu} \ge \prob{S \ge (1 + \eps)\nu}, \qquad S := \sum_{i=1}^n \ind{\{X_i = r\}} \sim \Bin(n, \zeta_r).
 \end{equation}
 Since $\zeta_r \to 0$ and $\E S = n\zeta_r \sim \nu \to \infty$, by a well-known large deviations fact (see \citer{DenHollander}{Exercise III.10} or \citer{AGSW23}{Proposition~8}) we obtain that
 \[
   \log \prob{S \ge (1 + \eps)\nu} \ge -(1+ o(1))\phi(\eps)\nu \ ,
 \]
 which together with~\eqref{eq:lower} gives the lower bound.
 \end{proof}

 \subsection{Large values}
 Further we bound the contribution of the large values. The conditions in the following Lemma look a bit technical, but simplify after noticing that $\log \frac{M}{Np} \asymp \min\left\{ \log n, \log \frac{1}{p} \right\}$. 
\begin{lemma}
  \label{lem:big_value_sum_LDP}
  Let  $M := \min \{ \left( r! \eps \nu \right)^{1/r}, N\}$. Assume that $p \to 0$ and a sequence of positive numbers $\eta = \eta(n) \to 0$ satisfies 
    \begin{equation}
      \label{eq:eta_log} \log \frac{1}{\eta} \ll \log \frac{M}{Np} 
    \end{equation}
    and
    \begin{equation}
      \label{eq:eta_M} \eta M \gg \frac{\log n}{\log \frac{M}{Np}}.
    \end{equation}
    Let $X_1, \dots, X_n$ be independent with distribution $\Bin(N,p)$. 
Then for every constant $\eps > 0$ random variable $Y'' := \sum_{i \in [n]} \binom{X_i}{r}\ind{\{X_i \ge \eta M\}}$ satisfies
    \begin{equation}
      \label{eq:big_value_UT}
      \log \prob{Y'' > \eps \nu} \le - (1 + o(1)) S \log \frac{M}{Np}     
    \end{equation}
   where 
    \[
      S = S(n,N,p,\eps) := \left( \floor{r!\eps \nu/N^r} + \left\{ r! \eps \nu/N^r \right\}^{1/r} \right) N  \ .
    \]
    If, instead, $X_1, \dots, X_n$ are degrees of $\G(n,p)$, then \eqref{eq:big_value_UT} holds with $N = n-1$ under an additional assumption 
    \begin{equation}
      \label{eq:Gnp_assumption}  
      \eta^3 \gg {\frac{\nu}{M^{r+1}}} \ .
    \end{equation}
\end{lemma}
\begin{proof}
  If $Y'' \ge \eps \nu$, then $\sum_{i \in [n]} X_i^r \ind{\{X_i > \eta M \}} \ge r!\eps \nu$. Therefore Proposition~\ref{prop:convex_sum} (with $f(x) = x^r$ and $t = r!\eps\nu$; note that $t \le n f(N)$ for large $n$ because of $p \to 0$) implies
  \begin{equation}
    \label{eq:sum_event}
  \sum_{i \in [n]}  X_i \ind{X_i > \eta M} \ge \left(\floor{r!\eps \nu/N^r} + \left\{ r! \eps \nu/N^r \right\}^{1/r}  \right)N = S.
\end{equation}
Event \eqref{eq:sum_event} is contained in the event that there is a set $J \subset [n]$ and numbers $y_j \ge \eta M, j \in J$ such that $\sum_{j \in J} y_j = S$ and for every $j \in J$ we have $X_j \ge y_j$. To make an efficient union bound over choices of $J$ and $\yy = (y_j : j \in J)$, let us observe that a few assumptions can be made. First, we can assume that numbers $y_j$ satisfy
\[
  y_j \le \min \left\{ S, N \right\} = \left. 
  \begin{cases} 
    (r!\eps \nu)^{1/r} & \text{ if } r! \eps \nu < N^r \\ 
    N &\text{ if } r! \eps \nu \ge N^r \end{cases} \right\} = M \ .
\] 
Note that $y_j \ge \eta M$ and $\sum_{j \in J} y_j = S$ implies $|J| \le S/(\eta M)$. By rounding down numbers $y_j$ to multiples of $\eta^2 M$, we might decrease the sum $\sum_{j \in J} y_j$ by at most $\eta^2 M \cdot \frac{S}{\eta M} = \eta S$ and therefore can assume that numbers $y_j \in (0, M]$, $j \in J$, are multiples of $\eta^2 M$ and satisfy $\sum_j y_j \ge (1 - \eta)S$. Define events $\cE_{J,\yy} := \left\{ X_j \ge y_j \text{ for every } j \in J \right\}$. 
Using that $y_j \ge \eta^2 M$ and \eqref{eq:eta_log}, inequality \eqref{eq:weak_Chernoff} implies
  \begin{equation}
    \label{eq:large_tail_}
    \log \prob{X_j \ge y_j} \le - y_j \log \frac{y_j}{\e Np} \le - y_j \left(\log \frac{M}{Np} + \log \frac{\eta^2}{\e} \right)\sim - y_j \log \frac{M}{Np}.
  \end{equation}
  Therefore independence of $X_1, \dots, X_n$, and $\sum_{j \in J} y_j \ge (1 - \eta)S \sim S$ imply
  \begin{equation}
    \label{eq:EJy}
    \log \prob{\cE_{J,\yy}} \le - (1 + o(1)) \sum_{j \in J} y_j \log \frac{M}{Np} \le - (1 + o(1)) S \log \frac{M}{Np} \ .
  \end{equation}
  There are at most $(n+1)^{S/(\eta M)}$ ways to choose the set $J$ of size at most $S/(\eta M)$ and at most $1/ \eta^2$ ways to choose $y_j \in (0,M]$ as a multiple of $\eta^2 M$. A union bound  and \eqref{eq:EJy} imply
  \begin{align}
    \label{eq:Y_primes_upp} \prob{Y'' \ge \eps \nu} &\le (n+1)^{S/(\eta M)} \cdot \left( \frac{1}{\eta^2} \right)^{S/(\eta M)} \exp \left( - (1 + o(1))S \log \frac{M}{Np} \right) \\
    \label{eq:Y_primes_intermediate} &= \exp \left( \frac{S}{\eta M} \log \frac{n+1}{\eta^2} - (1 + o(1))S \log \frac{M}{Np} \right) \ .
  \end{align}
  To show that the positive term in \eqref{eq:Y_primes_intermediate} is negligible we use conditions \eqref{eq:eta_log} and \eqref{eq:eta_M} to infer
  \begin{align*}
    \frac{1}{\eta M}\log \frac{n + 1}{\eta^2} & = \frac{\log (n+1) }{\eta M} + \frac{2\log \frac{1}{\eta}}{\eta M} \ll \log \frac{M}{Np} + \frac{\log \frac{M}{Np}}{\eta M} \sim \log \frac{M}{Np} \ ,
  \end{align*}
  which, combined with \eqref{eq:Y_primes_intermediate}, gives \eqref{eq:big_value_UT}.

  Finally we consider the setting when $X_1, \dots, X_n$ are degrees of $\G(n,p)$. The only place in the proof where the independence of $X_i$s was used was inequality \eqref{eq:EJy}. We redefine $\cE_{J,\yy}$ to be a more likely event that every vertex $j \in J$ has at least $y_j - |J|$ neighbors outside $J$ (which is an intersection of $|J|$ independent events). In view of the assumptions $y_j \ge \eta^2 M$ and $|J| \le S/(\eta M)$, inequalities \eqref{eq:EJy} remain valid once we check that $S/(\eta M) \ll \eta^2 M$. But this follows from the condition \eqref{eq:Gnp_assumption}, because $S \asymp \max \{\nu^{1/r}, \nu/N^{r-1}\} \asymp \nu M^{-(r-1)}$.

\end{proof}

\subsection{The whole sum in the sparse case}
The main result of this section is Theorem~\ref{thm:iid_UT} giving the asymptotics of the log-probability for i.i.d.\ binomial random variables in a sparse regime.
It comes with Proposition~\ref{prop:minimizers} which describes the minimizers in \eqref{eq:I_def_rep}. The proofs of Theorem~\ref{thm:iid_UT} and Proposition~\ref{prop:minimizers} are postponed until we state and prove a couple of auxiliary facts which are used in the proof of Theorem~\ref{thm:iid_UT}.

Recall from \eqref{eq:I_def} that for $\eps > 0$
  \begin{equation}
\label{eq:I_def_rep}
I_r(c, \eps) := 
\begin{cases}
  \min_{\delta \in [0, \eps]} \left( \phi(\eps - \delta) + \psi_r(\delta)c^{1/r - 1} \right) \quad &\text{if }c > 0 \\
  \phi(\eps) \quad & \text{if }c = 0
\end{cases}
  \end{equation}
The condition $p \le n^{o(1) - 1 - 1/r}$ in the following theorem could be relaxed, but it somewhat simplifies the proof as it allows to choose $R$ in Lemma~\ref{lem:small_value_sum_LDP} and $\eta$ in Lemma~\ref{lem:big_value_sum_LDP} so that $R = \eta M$. For higher $p$ we would have $R \ll \eta M$, which would require to show that the contribution from intermediate values in $(R, \eta M)$ is negligible.
\begin{theorem}
  \label{thm:iid_UT}
  Assume that $N = n-1$. If $n^{-1-1/r} \ll p \le n^{o(1)-1-1/r}$, then for every constant $\eps > 0$, with functions $\psi_r$ and $I_r$ defined in Theorem~\ref{thm:stars_UT},
  \begin{equation}
\label{eq:whole_UT}
-\log \prob{Y \ge (1 + \eps)\nu} \sim 
      \begin{cases}
	I_r(c,\eps)\nu  \quad &\text{if } \nu/\log^{r/(r-1)} n \to c \in [0, \infty) \\
	\psi_r(\eps)\nu^{1/r} \log n\quad &\text{if } \nu \gg \log^{r/(r-1)} n  
      \end{cases}\ .
  \end{equation}
\end{theorem}
\begin{proposition}
  \label{prop:minimizers}
  Fix $r \ge 2$ and let $\cM(c,\eps) = \left\{ \delta \in [0, \eps] : \phi(\eps - \delta) + \psi_r(\delta)c^{1/r-1} = I_r(c,\eps) \right\}$. For every constant $\eps > 0$ there is a constant $c_{r,\eps} > 0$ and a continuous strictly increasing function $\delta^*_{r,\eps} : [c_{r,\eps}, \infty) \to (0, \eps)$ satisfying $\lim_{c \to \infty} \delta^*_{r,\eps}(c) = \eps$ such that 
  \begin{equation*}
    \cM(c, \eps) = 
    \begin{cases}
      \{0\} \quad &\text{if } 0 \le c < c_r\\
      \{0, \delta^*_{r,\eps}(c_{r,\eps})\} \quad &\text{if } c = c_{r,\eps}\\
      \{\delta^*_{r,\eps}(c)\} \quad &\text{if } c > c_{r,\eps} \ .
    \end{cases}
  \end{equation*}
  Moreover, function $\eps \mapsto c_{r,\eps}$ is a decreasing bijection from $(0,\infty)$ to $(0, \infty)$.
\end{proposition}
\begin{lemma}
  \label{lem:cont_and_minimiz}
  For fixed $c \ge 0$, the function $\eps \mapsto I_r(c,\eps)$ is continuous and $\lim_{\eps \to \infty} I_r(c, \eps) = \infty$. For fixed $\eps > 0$, the function $c \mapsto I_r(c, \eps)$ is continuous on $[0, \infty)$ and $I_r(c, \eps) \sim \psi_r(\eps) c^{1/r-1}$ as $c \to \infty$.
\end{lemma}
\begin{proof}
  For $c = 0$ the first two claims are clear from the definition of $I_r(0,\eps) = \phi(\eps)$, hence further fix $c > 0$ and let $g_\eps(x) = \phi((1-x)\eps) + \psi_r(x \eps)  c^{1/r - 1}$ so that $I_r(c,\eps) = \min_{x \in [0, 1]} g_\eps(x)$. We claim that if $\eps_n \to \eps$, then $g_{\eps_n}$ converges uniformly to $g_\eps$ on $[0,1]$ and therefore $I_r(c,\eps_n) \to I_r(c,\eps)$, implying continuity in $\eps$. Using that $\phi'$ is continuous on $[0, \infty)$, we get that for every $y \in [0,1]$
  \begin{equation*}
    |\phi(\eps y) - \phi(\eps_n y)| \le |\eps y- \eps_n y| \max_{x \in [0,1]} |\phi'(x)| \le |\eps - \eps_n | \max_{x \in [0,1]} |\phi'(x)| \to 0 \ .
  \end{equation*}
  Moreover, for every $x \in [0,1]$,
\begin{equation*}
  |\left( x \eps \right)^{1/r} - \left( x \eps_n \right)^{1/r} | \le  x^{1/r} |\eps^{1/r} - \eps_n^{1/r}| \le  |\eps^{1/r} - \eps_n^{1/r}| \to 0\ .
\end{equation*}
  To see that $\lim_{\eps \to \infty} I_r(c,\eps) = \infty$, note that by monotonicity of functions $\phi$ and $\psi_r$
  \[
    g_{\eps}(\delta) \ge \phi(\eps/2) \ind{\{\delta \le \eps/2\}} + \psi_r(\eps/2)c^{1/r - 1} \ind{\{\delta > \eps\}} \ge \min \{\phi(\eps/2), \psi_r(\eps/2)c^{1/r - 1} \}, 
  \]
  which tends to infinity as $\eps$ grows.

  The properties of the function $c \mapsto I_r(c,\eps)$ follow from Proposition~\ref{prop:minimizers}: continuity is clear from continuity of $\phi, \psi_r$ and $\delta^*(c)$ in view of the description of $\cM(c,\eps)$; the asymptotics with respect to $c \to \infty$ hold because $I_r(c,\eps) \le \psi_r(\eps)c^{1/r-1}$ and, on the other hand, $I_r(c, \eps) \ge \psi_r(\delta^*(c))c^{1/r-1}$ and $\delta^*(c) \to \eps$.
\end{proof}
\begin{lemma}
  \label{lem:NA}
  For any real numbers $t, u$ we have \(\prob{Y' \ge t, Y'' \ge u}\le \prob{Y' \ge t}\prob{Y'' \ge u} \ .\)
\end{lemma}
\begin{proof}
  We use the notion of \emph{negatively associated} (\emph{NA}) random variables from \cite{JoagDevProschan}. For $i \in [n]$ let $X_i' = X_i \ind{X_i \le R}$ and $X_i'' = X_i \ind{X_i > R}$. 
  According to property P$_1$ in \cite{JoagDevProschan}, $X_i'$ and $X_i''$ are NA if
 \[
   \prob{X_i' \le x, X_i'' \le y} \le \prob{X_i' \le x}\prob{ X_i'' \le y}, \qquad\text{for all } x,y \in \R
 \]
 or, equivalently,
 \[
   \prob{X_i' > x, X_i'' > y} \le \prob{X_i' > x}\prob{ X_i'' > y}, \qquad\text{for all } x,y \in \R
 \]
 which is a trivial identity if $x < 0$ or $y < 0$ while the left-hand side is zero if both $x, y \ge 0$.
 By property P$_7$ in \cite{JoagDevProschan}, $X_1', X_1'', \dots, X_n', X_n''$ are NA, because they are a union of $n$ sets of independent NA variables. Note that $Y' = \sum_{i = 1}^n \binom{X_i'}{r}$ and $Y'' = \sum_{i = 1}^n \binom{X_i''}{r}$. Since for every number $t$ function $f_t(\xx) = \ind{\{\sum_{i = 1}^n \binom{x_i}{r} \ge t\}}$ is increasing, the definition of NA gives
  \begin{equation*}
   \prob{Y' \ge t, Y'' \ge u} - \prob{Y' \ge t}\prob{Y'' \ge u} = \Cov({f_t(X_1', \dots, X_n'), f_u(X_1'', \dots, X_n'')} ) \le 0 \ .
  \end{equation*}
\end{proof}

\begin{proof}[Proof of Theorem~\ref{thm:iid_UT}]
  Recall $M = \min \{(r!\eps \nu)^{1/r}, N\}$ from Lemma~\ref{lem:big_value_sum_LDP}. The conditions on~$p$ give 
  \begin{equation*}
    \log \frac{1}{Np} \asymp \log n, \quad 1 \ll M \sim (r!\eps \nu)^{1/r} = n^{o(1)} \quad \text{and} \quad \log \frac{M}{Np} \sim \frac{1}{r}\log n \ .
  \end{equation*}
  Choose $\eta = \eta(n) \to 0$ such that
 \[
   \frac{1}{M} \ll \eta \ll \frac{(\log n)^{1/(r-1)}}{M}
 \]
 and set $R := \eta M$. It is easy to check that such $\eta$ and $R$ satisfy the conditions of Lemmas~\ref{lem:small_value_sum_LDP} and~\ref{lem:big_value_sum_LDP}.

 \textit{Upper bound.} Recall the function $\psi_r(\delta) := r^{-1}(r!\delta)^{1/r}$. The upper bound on $p$ implies $\nu \ll N^r$; hence by Lemma~\ref{lem:big_value_sum_LDP}, for every constant $\delta \ge 0$ (note that for $\delta = 0$ it is trivial), we have
  \begin{equation}
    \label{eq:psi_hat_simpler}
    \log \prob{Y'' \ge \delta\nu} \le -(1 + o(1)) \psi_r(\delta) \nu^{1/r} \log n \ .
  \end{equation}
  Fix an integer $\ell$ such that $1/\ell < \eps$ and let $\eps_i := \frac{i}{\ell}\eps$ for $i = 0, \dots, \ell$. It follows that
  \begin{align*}
    \prob{Y \ge (1 + \eps)\nu} &= \sum_{i = 0}^{\ell - 2} \prob{Y'' \in \left[\eps_i\nu, \eps_{i+1}\nu\right), Y \ge (1 + \eps)\nu} + \prob{Y'' \ge \eps_{\ell - 1} \nu , Y \ge (1 + \eps)\nu} \\
    &\le \sum_{i=0}^{\ell - 2} \prob{Y'' \ge \eps_i \nu, Y' \ge (1 + \eps - \eps_{i+1}) \nu }  + \prob{Y'' \ge \eps_{\ell - 1} \nu}\\
    \justify{Lemma~\ref{lem:NA}, $\phi(0) = 0$}  &\le  \sum_{i=0}^{\ell - 2} \prob{Y' \ge (1 + \eps - \eps_{i+1})\nu} \prob{Y'' \ge \eps_i\nu} + \e^{-\phi(0)\nu} \cdot \prob{Y'' \ge \eps_{\ell - 1} \nu} \\ 
    \justify{Lemma~\ref{lem:small_value_sum_LDP} and \eqref{eq:psi_hat_simpler}} &\le \sum_{i=0}^{\ell - 1}  \exp \left( -(1 + o(1))\left[ \phi\left( \eps - \eps_{i+1} \right)\nu + \psi_r(\eps_i)\nu^{1/r} \log n \right] \right) \\
    \justify{$\ell$ is fixed, $\nu \to \infty$} &\le \exp \left( - (1 + o(1)) \min_{i = 0}^{\ell - 1} \left( \phi((\eps - 1/\ell - \eps_i))\nu    + \psi_r(\eps_i)\nu^{1/r}\log n \right)  \right) \ .
  \end{align*}
  Let $c_n := \nu / \log^{r/(r-1)} n$. Noting that $\nu^{1/r}\log n = \nu \cdot c_n^{1/r - 1}$, we obtain that the minimum in the last expression is at least $I\left( c_n, \eps - 1/\ell \right) \nu$. By Lemma~\ref{lem:cont_and_minimiz} this is asymptotically $I_r(c, \eps - 1/\ell) \nu$, when $c_n \to c \in [0, \infty)$ and asymptotically $\psi_r(\eps - 1/\ell) c_n^{1/r - 1} \nu = \psi_r(\eps - 1/\ell) \nu^{1/r} \log n$ when $c_n \to \infty$.  
  By Lemma~\ref{lem:cont_and_minimiz} function $\eps \mapsto I_r(c,\eps)$ is continuous and so is $\psi_r(\eps)$. Since $\ell$ can be chosen arbitrarily large, this completes the proof of the upper bound.

  \textit{Lower bound.} We will consider separately the vector $X_1, \dots, X_{n-1}$ and variable $X_n$ and then use their independence. 

  Fix an arbitrary constant $\delta \in [0, \eps)$. By applying Lemma~\ref{lem:small_value_sum_LDP} with $n-1$ instead of $n$ and $N = n-1$ and using the fact that $\nu \sim (n-1)\binom{n-1}{r}p^r$ and continuity of function $\phi$, we obtain
  \begin{equation*}
    \log \prob{\sum_{i = 1}^{n-1} \binom{X_i}{r}\ind{\{X_i \le R\}} \ge (1 + \eps - \delta)\nu} \ge  -(1 + o(1)) \phi(\eps - \delta) \nu\ .
  \end{equation*}
  Further, the condition on $p$ ensures that $\max\{1, np\} \ll \nu \ll n$ and therefore inequality \eqref{eq:binom_lower} implies 
  \begin{align*}
    \log \prob{\binom{X_n}{r} \ge \delta \nu} &\ge \log \prob{X_n \ge (r! \delta \nu)^{1/r}} \ge  -(1 + o(1))(r! \delta \nu)^{1/r} \log \frac{(r! \delta \nu)^{1/r}}{np} \\
  &\sim -(r! \delta \nu)^{1/r} \log n^{1/r} = -\psi_r(\delta) \nu^{1/r-1} (\log n) \cdot \nu \ .
  \end{align*}
  Independence and last two inequalities imply
  \begin{align}
    \nonumber \log \prob{Y \ge (1 + \eps)\nu} \ge &\log \prob{\sum_{i = 1}^{n-1} \binom{X_i}{r}\ind{\{X_i \le R\}} \ge (1 + \eps - \delta)\nu, \binom{X_n}{r} \ge \delta \nu} \\
    \label{eq:combined} &\ge - (1 + o(1)) \left[ \phi(\eps - \delta)  + \psi_r(\delta)\nu^{1/r - 1} \log n \right]  \nu
  \end{align}
  We complete the proof of the lower bound by case distinction. If $\nu/\log^{r/(r-1)} n \to 0$, we set $\delta = 0$, obtaining exponent $\phi(\eps) \nu = I_r(0,\eps)\nu$. If $\nu/\log^{r/(r-1)} n \to c \in (0, \infty)$, we set $\delta \in \cM(c, \eps)$ to be a minimizer (note that $\delta < \eps$ by Proposition~\ref{prop:minimizers}), which gives that \eqref{eq:combined} is asymptotically 
\[
  \left( \phi(\eps - \delta) + \psi_r(\delta) c^{1/r-1} \right) \nu = I_r(c,\eps) \nu \ .
\]
Finally, if $\nu/\log^{r/(r-1)} n \to \infty$, then \eqref{eq:combined} is asymptotically $\psi_r(\delta) \nu^{1/r} \log n$ and since we can pick $\delta$ arbitrarily close to $\eps$ the desired lower bound holds by continuity of $\psi_r$.
\end{proof}
\begin{proof}[Proof of Proposition~\ref{prop:minimizers}]
  Using change of variable 
  \begin{equation}
    \label{eq:alpha_c}
    \alpha = r^{-1}(r!)^{1/r}c^{1/r - 1}, \qquad c > 0 \ ,
  \end{equation}
    we define functions $f(\alpha, \delta, \eps) := \phi(\eps - \delta) + \alpha \delta^{1/r}$ so that
  \begin{equation}
    \label{eq:def_g}
    \cM(c,\eps) = \left\{ \delta \in [0, \eps] : f(\alpha, \delta, \eps) = \min_{\delta \in [0,\eps]} f(\alpha,\delta,\eps)\right\}, \quad \eps > 0.
  \end{equation}
  Fix $\eps > 0$ and let $f_\alpha(\delta) = f(\alpha, \delta, \eps)$.  Using that $\phi'(x) = \log (1 + x)$, we obtain that
\begin{equation}\label{eq:derivative}
  f_\alpha'(\delta) = 
  \frac{\alpha}{r \delta^{1 - 1/r}} - \log(1 + \epsilon - \delta), \qquad \delta \in (0, \eps].
\end{equation}
Since $g_\alpha(\delta) := \frac{\alpha}{r\delta^{1-1/r}}$ is strictly convex and $h(\delta) := \log\left( 1 + \eps - \delta \right)$ is strictly concave, function $f_\alpha' = g_\alpha - h$ is strictly convex. Let $\alpha_0(\eps) > 0$ be the value of $\alpha$ for which the graphs of $g_\alpha$ and $h$ touch, that is
\(
  \alpha_0(\eps) := \min \left\{ \alpha \ge 0 :  g_\alpha \ge h\right\} = \max \left\{ r\delta^{1 - 1/r}\log (1 + \eps - \delta) : \delta \in [0,\eps] \right\}.
  \)
For later purposes we observe that the latter expression it easily follows that
\begin{equation}
  \label{eq:alpha_zero_zero}
  \lim_{\eps \to 0} \alpha_0(\eps) = 0 \ .
\end{equation}
A brief look at the graphs (see Figure~\ref{fig:graphs}) makes plain that for $\alpha \in (0, \alpha_0(\eps)]$ function $f_\alpha$ has two stationary points $0 < \delta_-(\alpha, \eps) \le \delta_+(\alpha, \eps) < \eps$ which coincide if and only if $\alpha = \alpha_0(\eps)$, are continuous both in $\alpha$ as well as $\eps$ and
\begin{equation}
  \label{eq:delta_plus}
  \alpha \mapsto \delta_+(\alpha, \eps) \text{ is decreasing on } (0, \alpha_0(\eps)] \text{ with } \lim_{\alpha \to 0} \delta_+(\alpha, \eps) = \eps\ .
\end{equation}
On the other hand, for $\alpha > \alpha_0(\eps)$ function $f_\alpha$ has no stationary points.

\begin{figure}[t]
    \centering
    \begin{subfigure}[b]{0.49\textwidth}
        \centering
        \includegraphics[width=6cm]{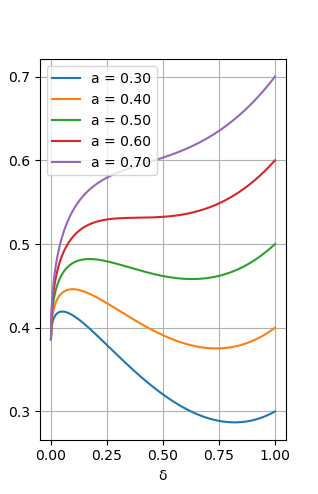} 
        \caption{}
        \label{fig:left}
    \end{subfigure}
    \hfill
    \begin{subfigure}[b]{0.49\textwidth}
        \centering
        \includegraphics[width=6cm]{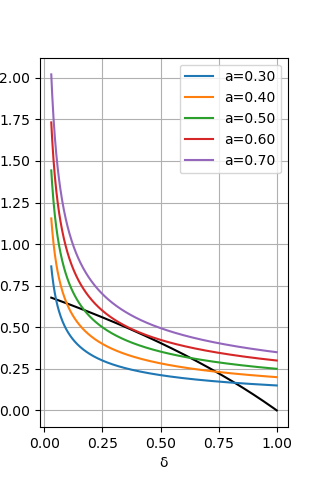} 
        \caption{}
        \label{fig:right}
    \end{subfigure}
    \caption{Case $r = 2, \eps = 1$. Left: graphs of $f_a$ for a few values of $a$; right: graphs of $\log(1 + \eps - \delta)$ as well as $(a/r)\delta^{-1/r}$.}
    \label{fig:graphs}
\end{figure}
Note that $\lim_{\delta \to 0} f_\alpha'(\delta) = +\infty$ and $f_\alpha'(\eps)  > 0$. Therefore for $\alpha \ge \alpha_0(\eps)$ function $f_\alpha$ is increasing and therefore attains its minimum $f_\alpha(0) = \phi(0)$ only at zero, while for $\alpha \in (0, \alpha_0(\eps))$
\[
  \min f_\alpha  = \min \left\{ \phi(0), F(\alpha, \eps) \right\} \ , \quad \text{with} \quad  F(\alpha, \eps) := f_\alpha(\delta_+(\alpha, \eps))= f(\alpha, \delta_+(\alpha, \eps), \eps).
\]
Since $\pd{f}{\delta}(\alpha, \delta_+(\alpha, \eps), \eps) = f_{\alpha}'(\delta_+(\alpha,\eps)) = 0$, using the chain rule, for $\alpha, \eps > 0$
\begin{equation}
  \label{eq:part_G_eps}
  \pd{F}{\eps}(\alpha, \eps) = \pd{f}{\eps}(\alpha, \delta_+(\alpha, \eps), \eps)  = \phi'(\eps-\delta_+(\alpha, \eps)) = \log (1+\eps - \delta_+(\alpha, \eps)) > 0 \ 
\end{equation}
and
\begin{equation}
  \label{eq:part_G_alpha}
  \pd{F}{\alpha}(\alpha, \eps) = \pd{f}{\alpha}(\alpha, \delta_+(\alpha, \eps), \eps)  = \delta_+(\alpha, \eps)^{1/r} > 0 \ .
\end{equation}
By \eqref{eq:part_G_alpha}, $\alpha \mapsto F(\alpha, \eps)$ is strictly increasing. Recalling that $f_{\alpha_0(\eps)}$ is increasing, we get that $F(\alpha_0(\eps), \eps) = f_{\alpha_0(\eps)}(\delta_+(\alpha, \eps)) > \phi(\eps)$. On the other hand $\lim_{\alpha \to 0} F(\alpha, \eps) = 0$, because $F(\alpha, \eps) \le f(\alpha, \eps, \eps) = \alpha \eps^{1/r}$. We conclude that there exists a unique $\alpha_1(\eps) \in (0, \alpha_0(\eps))$ such that 
\begin{equation}
\label{eq:alpha_one}
F(\alpha_1(\eps), \eps) = \phi(\eps) \ .
\end{equation}
Recalling \eqref{eq:def_g}, we have shown
\begin{equation}
  \label{eq:minimizers_g}
  \cM(c, \eps) =  
    \begin{cases} 
    \{ \delta_+(\alpha, \eps) \} \quad &\text{if } 0 < \alpha < \alpha_1(\eps)\\
    \{ 0, \delta_+(\alpha, \eps)\} \quad &\text{if } \alpha = \alpha_1(\eps)\\
    \{ 0 \} \quad &\text{if } \alpha > \alpha_1(\eps) \ .
    \end{cases}
\end{equation}
Let us show that $\alpha_1 : (0, \infty) \to (0, \infty)$ is an increasing bijection.

Since $\alpha_1$ is defined by \eqref{eq:alpha_one} and both $\pd{(F-\phi)}{\eps}$ and $\pd{(F-\phi)}{\alpha}$ are continuous functions on $(0,\infty)\times(0,\infty)$, implicit differentiation yields, using \eqref{eq:part_G_eps} and \eqref{eq:part_G_alpha},
\[
  \alpha_1'(\eps) = -\frac{\pd{(F - \phi)}{\eps}(\alpha_1(\eps), \eps)}{\pd{(F - \phi)}{\alpha}(\alpha_1(\eps), \eps)} = \frac{\phi'(\eps) - \pd{F}{\eps}(\alpha_1(\eps),\eps)}{\pd{F}{\alpha}(\alpha_1(\eps),\eps)} = \frac{\log (1 + \eps) - \log (1 + \eps - \delta_+(\alpha_1(\eps), \eps))}{\delta_+(\alpha_1(\eps),\eps)^{1/r}} > 0 \ ,
\]
that is, $\alpha_1$ is an increasing continuous function.
Recall that $\alpha_1(\eps) < \alpha_0(\eps)$. From \eqref{eq:alpha_zero_zero} we get $\lim_{\eps \to 0} \alpha_1(\eps) = 0$. Next we show that if $\eps \to \infty$, then $\alpha_1(\eps) \to \infty$. Suppose, for contradiction, that $\alpha_1(\eps) \uparrow A < \infty$. It is intuitive from Figure~\ref{fig:right} and not hard to show formally that $\delta_+(\alpha_1(\eps), \eps) \ge \delta_+(A, \eps) \to \infty$. Since $f_{\alpha_1(\eps)}'(\delta_+(\alpha_1(\eps),\eps)) = 0$, from \eqref{eq:derivative} it follows that $\eps - \delta_+(\alpha_1(\eps),\eps) \to 0$ and in particular that $\delta_+(\alpha_1(\eps), \eps) \sim \eps$. Consequently $\alpha_1'(\eps) = \frac{\log (1 + \eps) + o(1)}{(1 + o(1))\eps^{1/r}} \gg \eps^{-1/r}$. Since $\alpha_1$ can be expressed as an integral of $\alpha_1'$ and $\int_1^\infty \eps^{-1/r} d \eps$ diverges, we get a contradiction to the assumption that $\alpha_1$ is bounded.

Recalling \eqref{eq:alpha_c}, we treat $\alpha$ as a function of $c$. Note that $\alpha : (0, \infty) \to (0,\infty)$ is a decreasing continuous bijection and so is its inverse. Given $\eps > 0$, we set $c_{r,\eps} = \alpha^{-1}(\alpha_1(\eps))$ and $\delta^*_{r,\eps}(c) = \delta_+(\alpha(c), \eps)$. Description of $\cM(c,\eps)$ follows from \eqref{eq:minimizers_g}, while the desired properties of functions $\delta^*_{r,\eps}$ and $\eps \mapsto c_{r, \eps}$ follow from the corresponding properties of functions $\alpha \mapsto \delta_+(\alpha, \eps)$ and $\alpha_1$.

\end{proof}

\section{Proof of Theorem~\ref{thm:stars_UT}}

We treat separately two overlapping cases (i) $p \le n^{o(1)-1-1/r}$ (equivalent to $\log \mu \ll \log n$) and (ii) $\mu \ge \log^{7r}n$. 

\textbf{Case (i).} It holds that $\mu = n^{o(1)} \ll n^r$ so we are in one of the first two cases of \eqref{eq:stars_UT}. 

We will derive the upper bound from Theorem~\ref{thm:iid_UT} using the first inequality in \eqref{eq:bad_sequences} and inequality \eqref{eq:reduct_upper}. 
Let $\Phi(\eps)$ stand for the expression on the right-hand side of \eqref{eq:stars_UT} and note that $\Phi(\eps)$ is of order $\Phi_n$, defined in \eqref{eq:Phi_order}.  Let $\cR_n = \cR_{C,n}$ as defined in \eqref{eq:def_cR}, with $C = C(\eps) > 0$ large enough so that \eqref{eq:bad_sequences} implies 
  \begin{equation}
    \label{eq:cR_negl}
    \prdp{\overline{\cR_n}} \le \e^{-\Phi(\eps)} \ .
  \end{equation}
  In the rest of the proof we will use the fact that $\delta_n = \delta_{C,n}$ defined in \eqref{eq:def_deltas} satisfies $\delta_n \to 0$ (see \eqref{eq:delta_small}).
   Defining the upper tail event 
    \[
      \cS_{n}(p,\eps) := \left\{ \dd \in \{0, \dots, n-1\}^n : \tbin{d_1}{r} + \dots + \tbin{d_n}{r} \ge (1 + \eps) n \tbin{n-1}{r}p^r \right\} \ ,
    \]
    we have 
  \begin{align}
    \notag \prob{X \ge (1 + \eps)\mu} &= \prdp{\cS_n(p,\eps)} \le \prdp{\overline{\cR_n}} + \prdp{\cS_n(p,\eps) \cap \cR_n} \\
    \label{eq:Sn_upper} \justify{\eqref{eq:cR_negl},\eqref{eq:reduct_upper}} &\le \e^{ -\Phi(\eps) } + O(n^2p) \max_{p' = (1 \pm \delta_n)p} \prb{p'}{\left(\cS_n(p,\eps)\right)} \ .
  \end{align}
Now for $p' \le (1 + \delta_n)p$ we have 
\[
  (1 + \eps) n \binom{n-1}{r}p^r \ge (1 + \eps') n \binom{n-1}{r}(p')^r, \quad \text{with } \eps' := (1 + \eps)/(1 + \delta_n)^r - 1 \to \eps,
\]
whence $\cS_n(p,\eps) \subset \cS_n(p', \eps')$.
Now preparing to apply Theorem~\ref{thm:iid_UT} (in which $\nu = \mu$ due to $N = n-1$), assume that $\mu/\log^{r/(r-1)} n \to c \in [0, \infty]$. Since $p' \sim p$, we get hat also
\[
  n\binom{n-1}{r}(p')^r / \log^{r/(r-1)} n \to c.
\]
Moreover, the right-hand side of \eqref{eq:whole_UT} is a continuous function of $\eps$ in each of the two cases (when $c \in [0, \infty)$ this follows from Lemma~\ref{lem:cont_and_minimiz}). In view of the last two observations Theorem~\ref{thm:iid_UT} implies that $\log \prb{p'}{\left(\cS_n(p',\eps')\right)} \le (1 + o(1)) \Phi(\eps)$. Combining this with \eqref{eq:Sn_upper} and noting that the factor $O(n^2p)$ is $\e^{O(\log n)} = \e^{o(\Phi(\eps))}$ due to the assumption $\mu \gg \log n$, we have shown the upper bound.

For the lower bound we apply inequality \eqref{eq:reduct_lower} with $\cS_n = \cS_n(p,\eps)$ and Theorem~\ref{thm:iid_UT}, which implies $\prbp{\cS_n} \ge \e^{-(1 + o(1))\Phi(\eps)}$. By choosing $C = C(\eps)$ in the definition of $\cR_n = \cR_{C,n}$ large enough, we can assume, by \eqref{eq:bad_sequences}, that $\prbp{\overline{\cR_n}} \ll \prbp{\cS_n}$. Recalling that $\Phi(\eps) \asymp \Phi_n$, we will complete the proof in the case (i) by showing that
\begin{equation}
  \label{eq:cT_upper}
    \prbp{\cR_n \cap \overline{\cT_n}} \le \e^{-\omega(\Phi_n)} \ .
\end{equation}
To prove \eqref{eq:cT_upper}, we apply Theorem~\ref{thm:iid_UT} with $r = 2$. If $\dd \in \cR_n \cap \overline{\cT_n}$, then
\[
  2n^2p < \sum_{i \in [n]} d_i^2 = 2\sum_{i \in [n]} \binom{d_i}{2} +  \sum_{i \in [n]} d_i \le 2\sum_{i \in [n]} \binom{d_i}{2} + (1 + \delta_n)n^2p  
\]
and therefore, using $pn \to 0$ 
\begin{equation*}
\sum_{i \in [n]} \binom{d_i}{2} \ge \frac{1 - \delta_n}{2}n^2p \gg n^3p^2.
\end{equation*}
That is, the upper tail event holds with arbitrarily large $\eps$.
For $r = 2$, the right-hand side of \eqref{eq:whole_UT} is of order 
\[
  \min\{n(np)^2, n^{1/2} (np) \log n\} \geBy{ \text{$np \to 0$} } \min\{n(np)^r, n^{1/r}(np) \log n \} = \Phi_n
\]
  while the coefficient depending on $\eps$ tends to infinity as $\eps \to \infty$ (in the case $c \in (0,\infty)$ it follows from Lemma~\ref{lem:cont_and_minimiz}). Therefore by Theorem~\ref{thm:iid_UT} inequality \eqref{eq:cT_upper} holds.

\textbf{Case (ii)} overlaps with the three last cases of \eqref{eq:stars_UT} and we again denote by $\Phi(\eps)$ the right hand side of \eqref{eq:stars_UT} and will use that $\Phi(\eps) \asymp \Phi_n$ for any fixed $\eps > 0$. To prove the upper bound we will use decomposition $X = Y' + Y''$ with $Y',Y''$ defined in \eqref{eq:Y_primes} (but with $X_1, \dots, X_n$ being the vertex degrees of $\G(n,p)$) and apply Theorem~\ref{thm:C} to $Y'$ and the last claim of Lemma~\ref{lem:big_value_sum_LDP} to $Y''$. Note that since $N = n-1$, we have $\nu = \mu)$. Consistently with the notation of Lemma~\ref{lem:big_value_sum_LDP}, we set
\[
  R := \eta M, \quad \text{ with } \quad M := \min\left\{ (r!\eps \mu)^{1/r}, n-1 \right\}, \quad \eta := {\log^{-2} (1/p)} \ .
\]

To apply Threorem~\ref{thm:C}, let $\fS$ consist of all $\binom{n}{2}$ pairs of vertices in $[n]$, and let $\cI$ be the family of $r$-element subsets which form the edge-sets of $r$-stars. Hence $|\cI| = n \binom{n-1}{r}$, while $Y_\alpha$ is the indicator that star $\alpha$ is present in $\G(n,p)$ and $\sum_{\alpha \in \cI} \E Y_\alpha = n\binom{n-1}{r}p^r = \mu$.

Given a graph $G$ on vertex set $[n]$, we say that a $r$-star is \emph{dull} if its center vertex has degree at most $R$. Given a dull star, it shares edges with at most $2r\binom{R-1}{r-1}$ other dull stars. Hence the subfamily $\cJ_G \subset \cI$ corresponding to dull stars in $G$ satisfies \eqref{eq:max_cond} with 
\[
  C := 2r\binom{R-1}{r-1}\le 4R^{r-1} = \frac{4 M^{r-1}}{\log^{2(r-1)} (1/p)}.
\]
Consider $\alpha \in \cI$. If $\alpha$ is a dull star in $\G(n,p)$, that is, $\alpha \in \cJ_{\G(n,p)}$, then in particular it is a star in $\G(n,p)$, that is, $Y_\alpha = 1$.
Since $Y'$ counts the number of dull $r$-stars in $\G(n,p)$, it follows that 
\[
  Y' = \sum_{\alpha \in \cJ_{\G(n,p)}} 1 = \sum_{\alpha \in \cJ_{\G(n,p)}} Y_\alpha \le Z_C,
\]
and thus $\prob{Y' \ge (1 + \delta)\mu} \le \prob{Z_C \ge (1 + \delta)\mu}$ for any $\delta$.
  Theorem~\ref{thm:C} implies that for every $\delta > 0$
\begin{align}
  \notag -\log \prob{Y' \ge (1 + \delta) \mu} &\ge - \log \prob{Z_C \ge (1 + \delta) \mu} \ge \frac{\phi(\delta)\mu}{C} \ge  \frac{\phi(\delta)\mu \log^{2(r-1)} (1/p)}{4M^{r-1}} \\
  \label{eq:dull_upper} &\asymp \max \{\mu^{1/r}, \mu/n^{r-1}  \} \log^{2(r-1)} (1/p) \gg \Phi_n \ .
\end{align}
To bound the tail of $Y''$ we will apply the last assertion of Lemma~\ref{lem:big_value_sum_LDP}. In order to verify that conditions \eqref{eq:eta_log}, \eqref{eq:eta_M} and \eqref{eq:Gnp_assumption} are satisfied, it is useful to note that $\log \frac{M}{(n-1)p} \sim \min \{ \log n^{1/r}, \log \frac{1}{p} \} \asymp \log \frac{1}{p} =: \ell.$
First, \eqref{eq:eta_log} holds because
\[
\log \frac{1}{\eta} = \log (\ell^2) \ll \ell \asymp \log \frac{M}{(n-1)p} \ .
\]
Next, \eqref{eq:eta_M} holds, because, using $\mu \ge \log^{7r} n$,
\[
  \eta M \asymp \frac{\min \{\mu^{1/r}, n\}}{\ell^2} \gg \frac{\log n}{\ell} \asymp \frac{\log n}{\log \frac{M}{(n-1)p}} \ .
\]
Finally, \eqref{eq:Gnp_assumption} holds, because
\[
  \frac{\mu}{M^{r+1}} \asymp \max \left\{ \mu^{-1/r}, p^r \right\} \le \max \left\{ 1/\log^7 n, \e^{-r \ell} \right\} \ll {\ell^{-6}} = \eta^3 \ .
\]
 Given an arbitrary constant $\delta \in (0, \eps)$, elementary calculations show that $S(n,n-1,p, \eps - \delta) \sim \Phi(\eps - \delta)$. Hence inequality \eqref{eq:dull_upper} and Lemma~\ref{lem:big_value_sum_LDP} imply
\begin{align*}
  \prob{X \ge (1 + \eps) \mu} &\le \prob{Y' \ge (1 + \delta)\mu} + \prob{Y'' \ge (\eps - \delta)\mu}\\
  &\le \e^{-\omega(\Phi_n)} + \e^{ - (1 + o(1))\Phi(\eps - \delta) } = \e^{- (1 + o(1))\Phi(\eps - \delta)}\ .
\end{align*}
Since $\delta$ can be chosen arbitrarily small, and $\Phi(\eps)$ depends continuously on $\eps$, the upper bound follows.

For the lower bound assume $\mu/\binom{n}{r} \to \rho \in [0, \infty]$ where $\rho = 0$ and $\rho = \infty$ correspond to the second and fourth cases in \eqref{eq:stars_UT}. Choose a `small' constant $\delta > 0$. Let 
\[
  a := \floor{(\eps + \delta) \mu/\tbin{n}{r}} \quad \text{and} \quad
  b:=  
    \begin{cases}
      \left( r!(\eps + \delta)\mu \right)^{1/r} \quad &\text{if } \rho = 0\\
      \left\{ (\eps + \delta) \rho \right\}^{1/r}(n-1) \quad &\text{if } \rho \in (0, \infty)\\
      n-1 \quad &\text{if } \rho = \infty \ .
    \end{cases}
\]
Note that $p \to 0$ implies $a \ll n$ and define event
\[
  \cA := \{X_1 = \dots = X_a = n-1, X_{a+1} \ge b\}.
\]
If $\delta \in (0, \infty)$, for small enough $\delta$ we have that $(\eps + \delta)\rho \notin \Z$ and therefore $a \to \floor{(\eps + \delta) \rho}$ whenever $\rho \in [0,\infty)$.
With this in hand, it is straightforward to verify that event $\cA$ implies $\sum_{i \in [a+1]}\binom{X_i}{r} \ge a \binom{n-1}{r} + \binom{b}{r} \sim (\eps + \delta)\mu$. 
Further let $\hat X$ be the number of $r$-stars in the subgraph of $\G(n,p)$ induced by $[n] \setminus [a+1]$. Using that $a \ll n$, we obtain that  
\[
  \hat \mu := \E \hat{X} = (n-a-1)\binom{n-a-2}{r}p^r \sim \mu\ .
\]
Since $\hat{X}$ is independent of $(X_i : i \in [a+1])$, we have, for $n$ large enough,
\[
  \prob{X \ge (1 + \eps) \mu} \ge \prob{\cA}\prob{\hat{X} \ge (1 - \delta/2)\hat{\mu}} \ .
\]
It is known (see, e.g., \citer{JLRbook}{Lemma~3.5}) that $\Var \hat{X} \asymp \hat{\mu}^2/(\min_{k = 1}^r n^{k+1}p^k) \ll \hat{\mu}^2$. Hence Chebyshev's inequality easily implies that $\prob{\hat{X} \ge (1 - \delta/2)\hat{\mu}} \to 1$. It remains to prove that
\begin{equation}
  \label{eq:lower_goal}
  \log \prob{\cA} \ge -(1 + o(1)) \Phi(\eps + \delta),
\end{equation}
as then the fact that $\Phi(\eps + \delta) \sim \Phi(\eps)$ for $\delta \to 0$ gives the desired lower bound.
Since $a \ll n$, we get 
  \begin{align*}
    \prob{\cA} &= p^{a(n-1) - \binom{a}{2}}\prob{\Bin(n-1 -a,p) + a \ge b} \\
    &\ge p^{an(1 + o(1))}\prob{\Bin(n-1,p) \ge b}.
  \end{align*}
  Recall that for $\rho \in (0, \infty)$ we can assume $(\eps + \delta)\rho \notin \Z$. It follows that $b \gg np$. By \eqref{eq:binom_lower} we get
  \begin{align*}
    \log \prob{\cA} &\ge -(1 + o(1)) \left[na\log \frac{1}{p} + \ceil{b}\log \frac{\ceil{b}}{(n-1)p} \right] \\
  &\sim - 
    \begin{cases}
      0 + (r!(\eps + \delta)\mu)^{1/r} \log n^{1/r}   \quad &\text{if } \rho = 0\\
      \left( \floor{(\eps + \delta)\rho} + \left\{ (\eps + \delta)\rho \right\}^{1/r}\right) n \log \frac{1}{p} \quad &\text{if } \rho \in (0, \infty)\\
      n(\eps + \delta)\mu/\binom{n}{r} \log \frac{1}{p} \quad &\text{if } \rho = \infty \ .
    \end{cases}
  \end{align*}
 Since $\rho \in (0, \infty)$ implies that $p = n^{-1/r + o(1)}$ and therefore $\log (1/p) \sim r^{-1} \log n$, inequality \eqref{eq:lower_goal} follows, thus completing the proof of Case (ii) and thus the whole theorem. \qed

\bibliographystyle{plain}
\bibliography{star_tails}

\begin{bibdiv}
\begin{biblist}

\bib{AGSW23}{article}{
      author={Ara{\'u}jo, Pedro},
      author={Griffiths, Simon},
      author={{\v{S}}ileikis, Matas},
      author={Warnke, Lutz},
       title={Extreme local statistics in random graphs: maximum tree extension
  counts},
        date={2023},
         url={https://arxiv.org/abs/2310.11661},
        note={Preprint, {arXiv}:2310.11661},
}

\bib{BB23}{article}{
      author={Basak, Anirban},
      author={Basu, Riddhipratim},
       title={Upper tail large deviations of regular subgraph counts in
  {Erd{\H{o}}s}-{R{\'e}nyi} graphs in the full localized regime},
        date={2023},
        ISSN={0010-3640},
     journal={Commun. Pure Appl. Math.},
      volume={76},
      number={1},
       pages={3\ndash 72},
}

\bib{BGLZ}{article}{
      author={Bhattacharya, Bhaswar~B.},
      author={Ganguly, Shirshendu},
      author={Lubetzky, Eyal},
      author={Zhao, Yufei},
       title={Upper tails and independence polynomials in random graphs},
        date={2017},
        ISSN={0001-8708},
     journal={Adv. Math.},
      volume={319},
       pages={313\ndash 347},
}

\bib{BC}{article}{
      author={Bordenave, Charles},
      author={Caputo, Pietro},
       title={Large deviations of empirical neighborhood distribution in sparse
  random graphs},
        date={2015},
        ISSN={0178-8051},
     journal={Probab. Theory Relat. Fields},
      volume={163},
      number={1-2},
       pages={149\ndash 222},
}

\bib{CV11}{article}{
      author={Chatterjee, Sourav},
      author={Varadhan, S. R.~S.},
       title={The large deviation principle for the {Erd{\H{o}}s}-{R{\'e}nyi}
  random graph},
        date={2011},
        ISSN={0195-6698},
     journal={Eur. J. Comb.},
      volume={32},
      number={7},
       pages={1000\ndash 1017},
}

\bib{AHMS}{misc}{
      author={Cohen~Antonir, Asaf},
      author={Harel, Matan},
      author={Mousset, Frank},
      author={Samotij, Wojciech},
        date={2024},
        note={Private communication.},
}

\bib{CDP24}{article}{
      author={Cook, Nicholas~A.},
      author={Dembo, Amir},
      author={Pham, Huy~Tuan},
       title={Regularity method and large deviation principles for the
  {Erd{\H{o}}s}-{R{\'e}nyi} hypergraph},
        date={2024},
        ISSN={0012-7094},
     journal={Duke Math. J.},
      volume={173},
      number={5},
       pages={873\ndash 946},
}

\bib{DenHollander}{book}{
      author={Den~Hollander, Frank},
       title={Large deviations},
      series={Fields Inst. Monogr.},
   publisher={Providence, RI: AMS, American Mathematical Society},
        date={2000},
      volume={14},
        ISBN={0-8218-1989-5},
}

\bib{DAM}{article}{
      author={Doku-Amponsah, Kwabena},
      author={M{\"o}rters, Peter},
       title={Large deviation principles for empirical measures of colored
  random graphs},
    language={English},
        ISSN={1050-5164},
     journal={Ann. Appl. Probab.},
      volume={20},
      number={6},
       pages={1989\ndash 2021},
}

\bib{HMS22}{article}{
      author={Harel, Matan},
      author={Mousset, Frank},
      author={Samotij, Wojciech},
       title={Upper tails via high moments and entropic stability},
        date={2022},
        ISSN={0012-7094},
     journal={Duke Math. J.},
      volume={171},
      number={10},
       pages={2089\ndash 2192},
}

\bib{JLRbook}{book}{
      author={Janson, Svante},
      author={{\L}uczak, Tomasz},
      author={Ruci{\'n}ski, Andrzej},
       title={Random graphs},
      series={Wiley-Intersci. Ser. Discrete Math. Optim.},
   publisher={New York, NY: Wiley},
        date={2000},
        ISBN={0-471-17541-2},
}

\bib{JRdeletion}{article}{
      author={Janson, Svante},
      author={Ruci{\'n}ski, Andrzej},
       title={The deletion method for upper tail estimates},
        date={2004},
        ISSN={0209-9683},
     journal={Combinatorica},
      volume={24},
      number={4},
       pages={615\ndash 640},
}

\bib{JoagDevProschan}{article}{
      author={Joag-Dev, Kumar},
      author={Proschan, Frank},
       title={Negative association of random variables, with applications},
        date={1983},
        ISSN={0090-5364},
     journal={Ann. Stat.},
      volume={11},
       pages={286\ndash 295},
}

\bib{M85}{article}{
      author={McKay, Brendan~D.},
       title={Asymptotics for symmetric 0-1-matrices with prescribed row sums},
        date={1985},
        ISSN={0381-7032},
     journal={Ars Comb.},
      volume={19A},
       pages={15\ndash 25},
}

\bib{MW91}{article}{
      author={McKay, Brendan~D.},
      author={Wormald, Nicholas~C.},
       title={Asymptotic enumeration by degree sequence of graphs with degrees
  $o(n^{1/2})$},
        date={1991},
        ISSN={0209-9683},
     journal={Combinatorica},
      volume={11},
      number={4},
       pages={369\ndash 382},
}

\bib{SWcounterexample}{article}{
      author={{\v{S}}ileikis, Matas},
      author={Warnke, Lutz},
       title={A counterexample to the {DeMarco}-{Kahn} upper tail conjecture},
        date={2019},
        ISSN={1042-9832},
     journal={Random Struct. Algorithms},
      volume={55},
      number={4},
       pages={775\ndash 794},
}

\bib{SW20}{article}{
      author={{\v{S}}ileikis, Matas},
      author={Warnke, Lutz},
       title={Upper tail bounds for stars},
        date={2020},
        ISSN={1077-8926},
     journal={Electron. J. Comb.},
      volume={27},
      number={1},
       pages={paper P1.67},
         url={www.combinatorics.org/ojs/index.php/eljc/article/view/v27i1p67},
}

\bib{W17}{article}{
      author={Warnke, Lutz},
       title={Upper tails for arithmetic progressions in random subsets},
        date={2017},
        ISSN={0021-2172},
     journal={Isr. J. Math.},
      volume={221},
      number={1},
       pages={317\ndash 365},
}

\end{biblist}
\end{bibdiv}

\end{document}